\documentclass[twoside]{amsart}
\usepackage[latin9]{inputenc}
\usepackage{amsthm}
\usepackage{setspace}
\usepackage{amssymb}
\numberwithin{equation}{section} 
\numberwithin{figure}{section} 
\theoremstyle{plain}
\theoremstyle{plain}
\newtheorem{thm}{Theorem}
  \theoremstyle{plain}
  \newtheorem{cor}[thm]{Corollary}
  \theoremstyle{definition}
  \newtheorem{defn}[thm]{Definition}
  \theoremstyle{remark}
  \newtheorem*{acknowledgement*}{Acknowledgement}
  \theoremstyle{plain}
  \newtheorem{lem}[thm]{Lemma}
  \theoremstyle{plain}
  \newtheorem*{algorithm*}{Algorithm}
  \theoremstyle{definition}
  \newtheorem*{problem*}{Problem}

\newcommand{\sys}[1]{X_{#1}}

\begin{document}

\title{A note on universality in multidimensional symbolic dynamics }

\author{Michael Hochman}


\curraddr{Fine Hall, Washington Road, Princeton University, Princeton, NJ 08544}

\email{hochman@math.princeton.edu}
\begin{abstract}
We show that in the category of effective $\mathbb{Z}$-dynamical
systems there is a universal system, i.e. one that factors onto every
other effective system. In particular, for $d\geq3$ there exist $d$-dimensional
shifts of finite type which are universal for $1$-dimensional subactions
of SFTs. On the other hand, we show that there is no universal effective
$\mathbb{Z}^{d}$-system for $d\geq2$, and in particular SFTs cannot
be universal for subactions of rank $\geq2$. As a consequence, a
decrease in entropy and Medvedev degree and periodic data are not
sufficient for a factor map to exists between SFTs.

We also discuss dynamics of cellular automata on their limit sets
and show that (except for the unavoidable presence of a periodic point)
they can model a large class of physical systems.
\end{abstract}
\maketitle
\markboth{Michael Hochman}{Universality in symbolic dynamics}

\section{Introduction}

\subsection{Universality for shifts of finite type}

A basic problem about any class of dynamical systems is to understand
the factoring relation between its members. Much of ergodic theory
and topological dynamics, and particularly the theory of one-dimensional
shifts of finite type (SFTs), has been motivated by the hope, which
for some classes is satisfied, that the factoring relation reduce
to some simple parameter, such as entropy, periodic point data or
spectrum. For higher dimensional SFTs, which are the main subject
of this note, partial results are known under certain mixing conditions
\cite{JK05}, but it has become progressively clearer that the invariants
which dictate the factoring relation in dimension $1$ are only a
part of the picture in dimensions $d>1$.

We begin by reviewing some definitions; see also section \ref{sec:Topological-and-symbolic-dynamics}.
A $\mathbb{Z}^{d}$ \emph{shift of finite type }(SFT) is a subshift
$X$ of the full $d$-dimensional shift $\Sigma^{\mathbb{Z}^{d}}$
over $\Sigma$, defined by excluding all configurations containing
patterns from some fixed finite set. More precisely, by a ($d$-dimensional)
pattern we mean a coloring of a finite subset of $\mathbb{Z}^{d}$.
For $F\subseteq\mathbb{Z}^{d}$ and $a\in\Sigma^{F}$, we say that
the pattern $a$ appears in a configuration $x\in\Sigma^{\mathbb{Z}^{d}}$
if $(T_{u}x)|_{F}=a$ for some $u\in\mathbb{Z}^{d}$; here, $T^{u}$
is the shift by $u$. For $L$ a set of $d$-dimensional patterns
over $\Sigma$, set \[
\sys{L}=\{x\in\Sigma^{\mathbb{Z}^{d}}\,:\,\textrm{no element of }L\textrm{ appears in }x\}\]
An SFT is a subset of the form $\sys{L}$ for a\emph{ finite} set
$L$. 

A \emph{subaction }(or subdynamic) of a $\mathbb{Z}^{d}$-SFT is the
restriction of the $\mathbb{Z}^{d}$-action to a subgroup $\mathbb{Z}^{k}<\mathbb{Z}^{d}$.
Thus the full action and its subactions share the same phase space,
which is a disconnected compact metric space, but it is important
to note that the subactions are not necessarily symbolic. For example,
if $\mathbb{Z}<\mathbb{Z}^{2}$ as the first component, then the $\mathbb{Z}$-subaction
of the full $\mathbb{Z}^{2}$-shift $\{0,1\}^{\mathbb{Z}^{2}}$ is
isomorphic to the $\mathbb{Z}$-shift over the Cantor set $\{0,1\}^{\mathbb{Z}}$,
which is not expansive; this may be seen by thinking of columns of
symbols as points in the Cantor set, so each 2-dimensional configuration
becomes a linear sequence of points in the Cantor set and the subaction
shifts these points. In particular a subshift of an SFT is may no
longer be an SFT; a finite-infinite subgroup always gives an SFT but
even the conditions under which a general subaction of an SFT is expansive
are not known. We use the unqualified term SFT to refer to an SFT
with the full action.

Returning to our subject, in this paper we consider a basic question
about the factoring relation for SFTs, namely, whether there is a
universal SFT that factors onto all others. We are actually interested
in the broader question of universality for the class of subactions
of SFTs. More precisely, for each $k\leq d$, we ask whether there
is a $\mathbb{Z}^{d}$-SFT $X$ so that the $\mathbb{Z}^{k}$-subaction
on $X$ factors onto the $\mathbb{Z}^{k}$-subaction of every other
$\mathbb{Z}^{d}$ SFT. Such an $X$, if it exists, we call a $(d,k$)-universal
SFT.

One can immediately rule out the existence of $(d,d)$-universal SFTs
on the basis of topological entropy, which does not increase upon
passage to a factor, and so, since every SFT has finite entropy but
there are SFTs of arbitrarily large entropy, no universal one can
exist. 

However, for $k<d$ it is not clear what one should expect. As we
saw in the case of the full shift, a subaction may have infinite entropy,
so this poses no restriction. There is a restriction of a recursive
nature, and that is that the subaction of an SFT is effective. Recall
that an \emph{effective $\mathbb{Z}^{k}$-dynamical system} (EDS)
is a subshift of the full $\mathbb{Z}^{k}$-shift over the Cantor
set whose complement is the union of a formally computable sequence
of basic open sets. The fact that a subaction of an SFT is effective
was proved in \cite{H07}, along with a partial converse: any effective
$\mathbb{Z}^{k}$ system can be realized as the factor of a $\mathbb{Z}^{k}$-subaction
of a $\mathbb{Z}^{k+2}$-SFT (in fact the extension can be made quite
small, but we will not use this). We refer to \cite{H07} for further
details.

Thus, the questions of whether $(d,k)$-universal SFTs exist is closely
related to the existence of universal systems in the class of effective
systems;%
\footnote{One must be careful what notion of morphism one chooses for effective
systems, since not every factor map is an effective factor map. However
in this paper both definitions lead to the same results, since a factor
map from an EDS to a symbolic system is automatically effective.%
} and the class of effective systems, though countable, includes essentially
every type of dynamics we can {}``describe''. This would seem to
indicate that we should not expect universal dynamics to exist, since
they do not for general systems (at least if we stay in the context
of separable spaces).

Another relevant piece of information was recently provided by S.
Simpson \cite{S07}, who introduced Medvedev degrees as an invariant
of SFTs. Simpson associates to each SFT $X$ the Medvedev degree $m(X)$
of its phase space, which is a measure of the recursive-theoretic
complexity of the phase space as a subset of the full shift without
reference to the dynamics (see section \ref{sec:Recursive-set-and-effective-dynamics}).
Since factor maps between SFTs are sliding block codes they are computable,
and therefore do not increase Medvedev degree. Simpson also showed
that every Medvedev degree is realized as a 2-dimensional SFT. It
follows that the factoring relation between SFTs is at least as complicated
as the order relation between Medvedev degrees, and the latter is
still not well understood. 

What is relevant to our question, however, is that there exists a
maximal Medvedev degree. Thus, from the point of view of recursion
theory, there is no obstruction to the existence of SFTs whose (sub)actions
factor onto a very broad class of systems; indeed, any set with maximal
Medvedev degree at least maps (via a computable function) into, every
effective set, and so into every SFT (this map has nothing to do with
dynamics, but even so its existence is non-trivial).

It turns out that the existence of universal effective systems depends
on the rank of the action. For $\mathbb{Z}$-actions, such a system
exists:
\begin{thm}
\label{thm:main-positive}There exists a universal effective $\mathbb{Z}$-system,
that is, an effective $\mathbb{Z}$-system that factors onto every
other effective $\mathbb{Z}$-system.

In particular, for every $d\geq3$ there exist $\mathbb{Z}^{d}$-SFTs
whose $\mathbb{Z}$-subaction factors onto the $\mathbb{Z}$-subaction
of every other $\mathbb{Z}^{d}$-SFT; i.e. there are $(d,1)$-universal
SFTs for every $d\geq3$.
\end{thm}
It remains an open problem whether there exist $(2,1)$-universal
SFTs, i.e. $\mathbb{Z}^{2}$-systems whose $1$-dimensional subactions
is universal. This would follow if every $\mathbb{Z}$-EDS could be
realized as the subaction of a $\mathbb{Z}^{2}$-SFT; this is problem
open \cite{H07}.

The universal effective $\mathbb{Z}$-system in the theorem is constructed
by taking the product of all non-empty $\mathbb{Z}$-EDS (this is
a countable product so no topological difficulties arise). The fact
that the non-empty $\mathbb{Z}$-EDS can be effectively enumerated
rests on the fact that one can decide whether a $\mathbb{Z}$-SFT
is empty based on the list of forbidden patterns which defines it.
In contrast, it is a classical result of Berger that this problem
is formally undecidable for $\mathbb{Z}^{2}$-SFTs \cite{R71,B66},
so this construction cannot be imitated in higher dimensions. This
is not a shortcoming of the method because
\begin{thm}
\label{thm:main-negative}If $d\geq2$ then there is no universal
effective $\mathbb{Z}^{d}$ system, and there are no $(d,k)$-universal
SFTs for $k\geq2$.
\end{thm}
To prove this, we show that the existence of a universal system could
be used as a component in an algorithm that would decide whether an
arbitrary finite set of patterns $L$ defines a non-empty SFT, which
would contradict Berger's theorem.

\subsection{Cellular automata}

Another class of dynamical systems defined by local rules are cellular
automata. Recall that a $d$-dimensional cellular automaton (CA) is
a function $f:\Sigma^{\mathbb{Z}^{d}}\rightarrow\Sigma^{\mathbb{Z}^{d}}$
which commutes with the shift action. This is well known to be equivalent
to being defined by a local, finite transition rule. See \cite{K05}
for definitions and a recent survey of the subject.

We next present two applications of the results from \cite{H07} to
the dynamics of CA. The first is analogous to the universality question
for SFTs, i.e.: are there universal CA? This questions seems to be
more difficult than for SFTs, except for the case $d=1$ where again
entropy considerations show that no universal object can exist. However,
using the relation between CA and SFTs (see e.g. \cite{H07}, section
5.1), we can show that for $d\geq3$ there exist CA whose dynamics
is very close to the universal effective $\mathbb{Z}$-system of theorem
\ref{thm:main-positive}. 

In discussing dynamics of CA one must first overcome the fact that
their action is in general neither invertible nor surjective (we are
interested in invertible dynamics, though the question makes sense
also in other categories). The limit set $\Lambda_{f}$ of a CA $f$
is the largest set on which $f$ acts surjectively: \[
\Lambda_{f}=\cap_{n=1}^{\infty}f^{n}(\Sigma^{\mathbb{Z}^{3}})\]
If $f$ act injectively on $\Lambda_{f}$ then the action is effective;
in any case, the natural extension of this system is effective.

Note that the limit action always contains a periodic point; this
imposes certain restrictions on the dynamics which can occur on limit
actions. However this is the only limitation. Combining theorem \ref{thm:main-positive}
with the results of \cite{H07} we have:
\begin{cor}
There exists a $3$-dimensional cellular automaton such that, after
removing from $\Lambda_{f}$ a fixed point and its basin of attraction,
is a universal $\mathbb{Z}$-EDS, and in particular factors onto the
natural extension of every CA.
\end{cor}
It is known that there are CA $f$ such that for any other CA $g$,
one can encode the configurations of $g$ into configurations of $f$
in a spatially homogeneous way and so that the action of $f$ simulates
the action of $g$ (for a precise definition see \cite{O02}). This
notion is not directly related to universality in our sense. 

Our second application concerns one of the motivations for studying
CA in the first place, namely that they provide a simple model for
evolution of physical systems. We would like to show that they live
up to this expectation in the sense that, for a reasonably large class
of such systems, we can find a CA which models their dynamics very
closely. We note that although much has been made of the fact that
CA can perform universal computation, this in itself does not say
much about their dynamics. The dynamics of a computer simulating a
dynamical system is quite distinct from the dynamics of the system
it is simulating. 

Since effective systems can be modeled as limit actions of CA (minus
the basin of attraction of a fixed point), we proceed by showing that
a large class of systems can be extended to EDS. This may be viewed
as an effective dynamical Hausdorff-Alexandroff theorem. For simplicity
we restrict our attention to attractors of maps of $\mathbb{R}^{n}$,
with the aim establishing it under reasonably simple.

The following definition is adapted from \cite{BS06} where it is
proposed as a natural model of effective computation over the reals.
A function $f$ defined on some subset of $\mathbb{R}^{d}$ is \emph{effective
}if there is an algorithm which, upon being given input $n\in\mathbb{N}$
and an the infinite array encoding the binary representations of $d$
numbers $x_{1},\ldots,x_{d}\in\mathbb{R}$, reads a finite number
of bits from the input and outputs $d$ rational numbers $y_{1},\ldots,y_{d}$
such that \[
\left\Vert f(x_{1},\ldots,x_{d})-(y_{1},\ldots,y_{d})\right\Vert _{\infty}\leq\frac{1}{n}\]

\begin{defn}
\label{def:effective attractor}Suppose that
\begin{enumerate}
\item $U\subseteq\mathbb{R}^{d}$ is an open set,
\item $f:U\rightarrow U$ is an effective map,
\item $X\subseteq U$ is a closed attractor for $f$, i.e. there is an open
set $V\subseteq U$ so that $\overline{f(V)}\subseteq V$ and $X=\cap f^{n}V$.
\item $f|_{X}$ is invertible.
\end{enumerate}
Then we say that $X$ is an effective attractor of $f$.
\end{defn}
As before, the presence of periodic points in the limit action of
a CA prevents CA from modeling arbitrary dynamics, but this is in
a sense the only obstruction.
\begin{thm}
\label{thm:realization-of-attractors}If $X\subseteq\mathbb{R}^{d}$
is an effective attractor for $f$ then there is a $3$-dimensional
cellular automaton $g$ such that, after removing from $\Lambda_{g}$
a fixed point and its basin of attraction, the action of $g$ factors
onto $(X,f)$. 
\end{thm}
For the proof one shows that effective attractors can be extended
to EDSs, and applies the machinery from \cite{H07}. As we have indicated,
the statement above is in a sense the best possible for invertible
dynamics without fixed points. It is true that in some ways, the dynamics
of the CA given by the theorem do not {}``look like'' the original
system: tracing back to the construction in \cite{H07} one sees that
the basin of attraction that we have thrown out is dense, and for
product measures on the configuration space typical points will converge
under the CA action to the fixed point. However, all the invariant
measures on the system, except for the point measure on the fixed
point, are pullbacks of measures from the original system. Furthermore,
under the technical assumption that the effective attractor to be
modeled satisfies the small boundary property, the EDS extending it
can be made to be injective on the complement of a universally null
set. Thus from the point of view of stationary dynamics the CA looks
very much like an extension of the original system.

\subsection*{Organization}

In section \ref{sec:Recursive-set-and-effective-dynamics} we recall
some of the recursive-theoretic machinery we will need and define
effective dynamics. In section \ref{sec:Universal-effective-Z-systems}
we prove theorem \ref{thm:main-positive}, and in section \ref{sec:Nonexistence-of-universal-REDS}
prove theorem \ref{thm:main-negative}. Section \ref{sec:Realization-of-effective-attractors}
discusses realization of effective attractors. Section \ref{sec:Open-Problems}
contains some open problems.
\begin{acknowledgement*}
I would like to thank Mike Boyle for some interesting discussions.
\end{acknowledgement*}

\section{\label{sec:Topological-and-symbolic-dynamics}Definitions and notation}

A $\mathbb{Z}^{d}$-dynamical system is the action of $\mathbb{Z}^{d}$
by homeomorphisms on a compact metric space $X$; the action of $u\in\mathbb{Z}^{d}$
is denoted usually by $T^{u}:X\rightarrow X$. A factor map between
systems $X,Y$ acted on by the same group is a continuous, onto map
$\pi:X\rightarrow Y$ which commutes with the action in the sense
that $\pi T^{u}=T^{u}\pi$ for every $u$.

Let $\Sigma$ be a finite set of symbols. The space $\Sigma^{\mathbb{Z}^{d}}$
of colorings of $\mathbb{Z}^{d}$ by $\Sigma$ is called the \emph{full
$d$-dimensional shift over $\Sigma$}, or just the full shift, and
its points are called \emph{configurations}. Topologically the full
shift is a Cantor set, and it comes equipped with a natural $\mathbb{Z}^{d}$
action, called the \emph{shift action}, in which $u\in\mathbb{Z}^{d}$
acts via the translation $T_{u}:\Sigma^{\mathbb{Z}^{d}}\rightarrow\Sigma^{\mathbb{Z}^{d}}$
defined by\[
(T^{u}x)(v)=x(u+v)\]
 A subset $X\subseteq\Sigma^{\mathbb{Z}^{d}}$ which is closed and
invariant to the shift (i.e. $T^{u}X=X$ for $u\in\mathbb{Z}^{d}$)
is called a \emph{subshift}, or a \emph{symbolic system}. By the Curtis-Hedlund-Lyndon
theorem \cite{HEDLUND69}, factor maps between subshifts of the same
dimension (but possibly distinct alphabets) are given by a \emph{block
code}: if $Y\subseteq\Delta^{\mathbb{Z}^{d}}$, $X\subseteq\Sigma^{\mathbb{Z}^{d}}$
and $\pi:Y\to X$ is a factor map, then there is a finite $F\subseteq\mathbb{Z}^{d}$
and a function $\pi_{0}:\Delta^{F}\rightarrow\Sigma$, so that $\pi$
acts on each site of $x\in\Delta^{\mathbb{Z}^{d}}$ by applying $\pi_{0}$
to the local neighborhood of the site: $(\pi x)(u)=\pi_{0}((T^{u}x)|_{F})$.
The diameter of $F$ is called the \emph{window size} of $\pi$. Conversely,
any such map $\pi_{0}:\Delta^{F}\rightarrow\Sigma$ gives rise to
a factor map $\pi$ in this way (the image is automatically a subshift).

\section{\label{sec:Recursive-set-and-effective-dynamics}Recursive sets and
effective dynamical systems}

\subsection{Some recursion theory}

We require some basic facts from recursion theory; see \cite{HU79}
for a formal introduction. Recursion theory provides a classification
of certain subsets of $\mathbb{N}$ according to the extent to which
the set may be described algorithmically. By an algorithm we mean
a finite set of instructions which can be carried out automatically,
i.e. by computer program or, more formally, a Turing machine. 

A subset $A\subseteq\mathbb{N}$ is \emph{recursive} (R) if there
is an algorithm which, given $n\in\mathbb{N}$, outputs {}``yes''
if $n\in A$ and {}``no'' otherwise. A function $f:\mathbb{N}\rightarrow\mathbb{N}$
is recursive if there is an algorithm which, given $n$, outputs $f(n)$.

A set $A$ is \emph{recursively enumerable }(RE) if there is an algorithm
which, on input $n$, returns {}``yes'' if $n\in A$ and otherwise
runs forever. Alternatively, a non-empty set $A\subseteq\mathbb{N}$
is RE if there is an algorithm which, given $n\in\mathbb{N}$, outputs
$a_{n}\in\mathbb{N}$ so that $A=\{a_{n}\,:\, n\in\mathbb{N}\}$;
in other words, it is the image of a recursive function. 

By fixing a bijection between $\mathbb{N}$ and another countable
set $U$, we can speak of R and RE subsets of $U$. Thus we will speak
of R and RE subsets of pairs of integers, finite sequences or patterns
over a finite set $\Sigma$, etc. We will always assume that the objects
have been placed in correspondence with $\mathbb{N}$ in some effective
way (for the purpose of classifying subsets as R or RE, two identifications
which can be algorithmically reduced to each other are equivalent).

Since there are countably many algorithms there are countably many
R and RE sets. We note that every recursive set is RE, but not vice-versa.
However, the examples of this tend to be rather artificial, e.g. the
set of provable theorems in number theory (G\"odel's theorem), or
the set of halting Turing machines (Turing's theorem). 

The following standard facts will be useful:
\begin{lem}
\label{lem:R-relation-of-RE-is-RE}Let $U$ be recursive and $L\subseteq U$
is an RE set. Let $R\subseteq U\times V$ be a recursive set and let\[
M=\{b\in V\,:\,(a,b)\in R\mbox{ for some }a\in L\}\]
Then $M$ is RE.\end{lem}
\begin{proof}
Let $A$ be an algorithm that on input $a\in U$ halts if $a\in L$
and runs forever otherwise. Let $B$ be the algorithm which, upon
input $b\in V$, iterates over all pairs $(n,a)\in\mathbb{N}\times U$,
and for each pair runs the algorithm $A$ for $n$ steps (or until
it halts) on the input $a$. If $A$ halts before $n$ steps are up,
it checks whether $(a,b)\in R$, and if so it halts; otherwise it
continues to the next pair $(n',a')$. It is easily seen that this
algorithm halts on input $b$ if and only if $b\in M$, so $M$ is
RE. \end{proof}
\begin{lem}
\label{lem:RE-and-coRE-isR}If a set $L\subseteq U$ if RE and $U\setminus L$
is RE then $L$ is R.\end{lem}
\begin{proof}
Let $A$ and $B$ be algorithms which, given $x\in U$, halt if $x\in L$
or $x\in U\setminus L$, respectively, and otherwise run forever.
Consider the algorithm which accepts $x$ as input and iterates over
$n\in\mathbb{N}$. For each $n$ it simulates $n$ steps of the computation
of $A$ on input $x$, and if that computation halted it announces
that $a\in L$ and halts. Otherwise it simulates $n$ steps of the
computation of $B$ on $x$ and if that computation halts, it announces
$x\notin L$ and halts. If neither simulations terminates, it proceeds
to the next $n$. Clearly, our algorithm always halts and gives the
correct answer.
\end{proof}

\subsection{Effective subshifts and EDS}

Returning to dynamics, let $K=\{0,1\}^{\mathbb{N}}$ denote the Cantor
set, and for finite $I\subseteq\mathbb{N}$ and $a\in\{0,1\}^{I}$
let $[a]$ denote the cylinder set determined by $a$, i.e. \[
[a]=\{x\in K\,:\, x(i)=a(i)\,,\, i\in I\}\]
it will be convenient to write $\mathcal{P}$ for the set of finite
patterns of the form $\{0,1\}^{I}$, $I\subseteq\mathbb{N}$; so $\mathcal{P}$
parameterizes the cylinder sets of $K$. The cylinder sets form a
closed and open basis for the topology of $K$.

Let $\Omega=\Omega_{d}=K^{\mathbb{Z}^{d}}$, which topologically is
again a Cantor set. A basis for the topology of $\Omega$ is given
by the generalized cylinder sets $[\overline{a}]$, where $\overline{a}:E\rightarrow\mathcal{P}$
for some finite $E\subseteq\mathbb{Z}^{d}$, and \[
[\overline{a}]=\prod_{u\in\mathbb{Z}^{d}}V_{u}\qquad\mbox{where }V_{u}=[a(u)]\mbox{ for }u\in E\mbox{ and }V_{u}=K\mbox{ otherwise}\]
We write $\mathcal{P}^{*d}$ for the set of such maps $\overline{a}:E\rightarrow\mathcal{P}$,
$E\subseteq\mathbb{Z}^{d}$ finite. The sets $[\overline{a}]$ for
$\overline{a}\in\mathcal{P}^{*d}$ form a closed and open basis for
$\Omega$.

As usual, let $\{T^{u}\}_{u\in\mathbb{Z}^{d}}$ denote the shift action
of $\mathbb{Z}^{d}$ on $\Omega$. A subshift $X\subseteq\Omega$
is, as usual, a nonempty, closed subset which is invariant under the
shift action. 

Every subshift (in fact every closed subset) is the complement of
a countable union of cylinder sets. If the complement is the union
of a recursive sequence of cylinder sets, we say it is effective.
To be precise, fix an effective enumeration of $\mathcal{P}$ and
use this to enumerate the elements of $\mathcal{P}^{*d}$, which parameterizes
a basis for $\Omega$. 
\begin{defn}
\label{def:effective-subshift}An effective subshift is a subshift
$X\subseteq\Omega$ such that $X=\Omega_{L}$ for some recursively
enumerable $L\subseteq\mathcal{P}^{*}$, where\begin{eqnarray*}
\Omega_{L} & = & \{x\in\Omega\,:\, x\notin T^{u}[\overline{a}]\mbox{ for every }a\in L\mbox{ and }u\in\mathbb{Z}^{d}\}\\
 & = & \Omega\setminus\bigcup_{\overline{a}\in L}\bigcup_{u\in\mathbb{Z}^{d}}T^{u}[\overline{a}]\end{eqnarray*}
or, equivalently, if the set\[
\{a\in\mathcal{P}^{*d}\,:\, X\cap[\overline{a}]=\emptyset\}\]
is recursively enumerable. 
\end{defn}
Since there are countably many algorithms, there are countably many
EDS, representing only countable many isomorphism types of systems.
In spite of this we do not know of any {}``natural'' invariant of
dynamical systems which cannot be realized as an EDS.

In \cite{H07} we showed:
\begin{thm}
\label{thm:sub-action-is-REDS}The subaction of an SFT is an EDS.
\end{thm}
More important for our present discussion is the partial converse
obtained there:
\begin{thm}
\label{thm:realizing-REDS}If $X$ is a $\mathbb{Z}^{d}$-EDS then
there is a $(d+2)$-SFT $Y$ and a $\mathbb{Z}^{d}$-subaction of
$Y$ which factors onto $(X,\mathbb{Z}^{d})$.
\end{thm}

\section{\label{sec:Universal-effective-Z-systems}Universal effective $\mathbb{Z}$-systems}

In this section we prove the following result:
\begin{thm}
\label{thm:universal-REDS}There exists a $\mathbb{Z}$-EDS which
factors onto every other $\mathbb{Z}$-EDS.
\end{thm}
Together with theorem \ref{thm:realizing-REDS} this proves theorem
\ref{thm:main-positive}.

A sequence $L_{n}\subseteq\mathbb{N}$ of sets is \emph{uniformly
recursively enumerable} if $\cup_{n=1}^{\infty}L_{n}\times\{n\}\subseteq\mathbb{N}^{2}$
is RE, or equivalently, if there is an algorithm $A$ which, given
input $i,j\in\mathbb{N}$, halts if $j\in L_{i}$ and otherwise runs
forever. 
\begin{lem}
\label{lem:RE-product-is-RE}Let $L_{n}\subseteq\mathcal{P}^{*d}$
be a uniformly RE sequence of sets. Write $X_{n}=\Omega_{L_{n}}$.
Then the product system $\prod_{n=1}^{\infty}X_{n}$ is an EDS.\end{lem}
\begin{proof}
The proof is based on the fact that $K\cong\prod_{n=1}^{\infty}K$,
and this isomorphism can be made effective. Fix some recursive bijection
$\varphi:\mathbb{N}^{2}\rightarrow\mathbb{N}$, and let \[
I_{n}=\{k\in\mathbb{N}\,:\, k=\varphi(i,n)\mbox{ for some }i\in\mathbb{N}\}\]
which is a partition of $\mathbb{N}$ into disjoint infinite sets.
We may identify $I_{n}$ with $\mathbb{N}$ via $\varphi(\cdot,n):\mathbb{N}\rightarrow I_{n}$.
Let $\pi_{n}:K\rightarrow K$ be the projection $x\mapsto x|_{I_{n}}$where
identify $x|_{I_{n}}$ with a point in $K$ using this bijection of
$I_{n}$ and $\mathbb{N}$. Extend $\pi_{n}$ to patterns over $K$
pointwise, so if $a\in K^{E}$ for a finite set $E\subseteq\mathbb{Z}^{d}$
then $(\pi_{n}(a))(u)=\pi_{n}(a(u))$, $u\in\mathbb{Z}^{d}$. Then
$x\mapsto(\pi_{1}(x),\pi_{2}(x),\ldots)$ is a homeomorphism from
$\Omega$ to $\prod_{n=1}^{\infty}\Omega$.

Let $L\subseteq\mathcal{P}^{*d}$ be defined by \[
L=\{\overline{a}\in\mathcal{P}^{*d}\;,\;\pi_{n}([\overline{a}])\subseteq[\overline{b}]\mbox{ for some }n\in\mathbb{N}\mbox{ and }\overline{b}\in L_{n}\}\]
A moments thought shows that $\Omega_{L}\cong\prod_{n=1}^{\infty}\Omega_{L_{n}}$.
In order to prove the lemma it suffices to show that $L$ is RE. Below
we provide the details.

First, note that if $a,b\in\mathcal{P}$ then we can check whether
$[a]\subseteq[b]$ in $K$: If $a:I\rightarrow\{0,1\}$ and $b:J\rightarrow\{0,1\}$
this amounts to verifying that $J\subseteq I$ and $b(i)=a(i)$ for
$i\in J$. 

Similarly, if $\overline{a},\overline{b}\in\mathcal{P}^{*d}$ then
we can decide whether $[\overline{a}]\subseteq[\overline{b}]$ in
$\Omega$: If $\overline{a}:E\rightarrow\mathcal{P}$ and $\overline{b}:F\rightarrow\mathcal{P}$,
we only need to check that $F\subseteq E$ and that $[\overline{a}(u)]\subseteq[\overline{b}(u)]$
for $u\in F$; by the above this inclusion is decidable.

Thus, if $\overline{a},\overline{b}\in\mathcal{P}^{*d}$ and $n\in\mathbb{N}$,
then we can decide whether $\pi_{n}([\overline{a}])\subseteq[\overline{b}]=\emptyset$
or not. 

The fact that $L$ is RE now follows from lemma \ref{lem:R-relation-of-RE-is-RE},
since by assumption $L_{*}=\cup_{n\in\mathbb{N}}\{n\}\times\{L_{n}\}$
is RE, and\[
L=\{\overline{a}\in\mathcal{P}^{*d}\;,\;\pi_{n}([\overline{a}])\subseteq[\overline{b}]\mbox{ for some }(n,\overline{b})\in L_{*}\}\qedhere\]

\end{proof}
We also need the following:
\begin{lem}
\label{lem:decidability-of-emptiness-of-1D-SFT}Given a finite $L\subseteq\mathcal{P}^{*1}$,
it is decidable whether $\Omega_{L}=\emptyset$ or not.\end{lem}
\begin{proof}
Suppose $L$ is given. Let $I\subseteq\mathbb{N}$ be large enough
that if $\overline{a}\in L$, $\overline{a}:E\rightarrow\mathcal{P}$,
then every pattern $\overline{a}(i)$, $i\in E$ is supported in $I$.
It is not hard to check that $\Omega_{L}\neq\emptyset$ if and only
if there is an infinite sequence $x=(x(n))_{n\in\mathbb{Z}}$ over
the alphabet $\{0,1\}^{I}$ so that, for every every $k\in\mathbb{N}$
there is no $\overline{a}\in L$, $\overline{a}:E\rightarrow\mathcal{P}$,
which satisfies $x(k+i)(n)=(\overline{a}(i))(n)$ for $i\in E$ and
all $n$ at which $\overline{a}(i)$ is defined. Thus, deciding whether
$\Omega_{L}$ is empty is equivalent to deciding whether a certain
$\mathbb{Z}$-SFT over the alphabet $\{0,1\}^{I}$ is empty (the last
condition, though cumbersome, is a finite, local restriction symbols
in $x$ and is equivalent to excluding a finite number of patterns).
Clearly the reduction from the first problem to the second is computable,
and the second problem is decidable, since it is equivalent to deciding
whether there are any cycles in a finite graph associated to the given
SFT in an effective way (see e.g. \cite{LM95}). Thus the original
problem is decidable as well. 
\end{proof}

\begin{proof}
(of theorem \ref{thm:universal-REDS}) We show that there is a uniformly
recursive sequence $L_{n}\subseteq\mathcal{P}^{*1}$ so that $X_{n}=\Omega_{L_{n}}\neq\emptyset$
for every $n$ and every $\mathbb{Z}$-EDS appears as one of the $X_{n}$'s.
Given such a sequence, the product $\prod_{n=1}^{\infty}X_{n}$ is
universal for EDS and is itself an EDS by the previous lemma.

Let $(A_{n})_{n=1}^{\infty}$ be a fixed recursive enumeration of
all algorithms and let \[
L'_{n}=\{\overline{b}\,:\, A_{n}\mbox{ halts on input }\overline{b}\}\]
The sequence $L'_{n}$ is uniformly RE, because given $n$ and $\overline{b}$,
in order to determine if $\overline{b}\in L'_{n}$ one first computes
$A_{n}$ (we can because the sequence $A_{n}$ is recursive) and then
simulates the computation of $A_{n}$ on input $\overline{b}$, halting
only if this computation halts. This achieves the first of our goals,
since clearly the sequence $\Omega_{L'_{n}}$ will contain all effective
subshifts. The problem is that some $\Omega_{L'_{n}}$'s will be empty.
We therefore will define an RE sequence $L_{n}$ with $L_{n}=L'_{n}$
if $\Omega_{L'_{n}}\neq\emptyset$, and with $\Omega_{L_{n}}\neq\emptyset$
in any case.

First, given $n,k\in M$ and $\overline{a}\in\mathcal{P}^{*d}$ we
say that $\overline{a}$ is $(n,k)$-recognized if the algorithm $A_{n}$
halts on input $\overline{a}$ within $k$ steps. Note that this condition
is recursive, since $A_{n}$ can be computed from $n$ and then its
computation on input $\overline{a}$ can be simulated for $k$ steps
to see if it halts.

Choose an enumeration $\overline{b}_{1},\overline{b}_{2},\ldots$
of $\mathcal{P}^{*d}$. Given $k$ let \[
L_{n,k}=\{\overline{b}_{i}\,:\,1\leq i\leq k\mbox{ and }\overline{b}_{i}\mbox{ is }(n,k)\mbox{-recognizable \}}\]
Clearly the $L_{n,k}$ can be computed given $n,k$, they are increasing
in $k$, and their union is $L'_{n}$. 

Finally, let\[
L_{n}=\{\overline{a}\in\mathcal{P}^{*d}\,:\,\Omega_{L_{n,k}}\neq\emptyset\mbox{ and }\overline{a}\in L_{n,k}\mbox{ for some }k\in\mathbb{N}\}\]
By lemma \ref{lem:R-relation-of-RE-is-RE} we see that $L_{n}$ is
RE. Also, if $\Omega_{L'_{n}}\neq\emptyset$ then $\Omega_{L_{n,k}}\neq\emptyset$
for every $k$ so $\overline{a}\in L_{n}$ if and only if $\overline{a}\in L'_{n}$,
or in other words, $L_{n}=L'_{n}$. On the other hand, if $\Omega_{L'_{n}}=\emptyset$
then by compactness there is a $k$ for which $\Omega_{L_{n,k}}=\emptyset$;
let $k_{0}$ be the minimal such $k$. One sees that $,\overline{a}\in L_{n}$
if and only if $\overline{a}\in L_{n,k_{0}-1}$, so $L_{n}=L_{n,k_{0}-1}$
and by definition $\Omega_{L_{n,k_{0}}}\neq\emptyset$. This completes
the proof.
\end{proof}

\section{\label{sec:Nonexistence-of-universal-REDS}Nonexistence of universal
SFTs}

The proof from the last section cannot be adapted to $\mathbb{Z}^{d}$-EDS
because, for $d\geq2$, one cannot decide in general if a $d$-dimensional
SFT is empty; this is Berger's theorem. Although this in itself is
not a proof that no $(d,k)$-universal SFTs exist for $k>1$, the
proof in fact involves showing that if one did exist then it could
be used as a component in an algorithm for deciding the emptiness
of SFTs, contradicting Berger's theorem. 

Although the proof that universal $\mathbb{Z}^{d}$-EDS don't exist
for $d\geq2$ is not conceptually difficult, it will be more transparent
to first establish the weaker claim that there are no $(d,d)$-universal
SFTs, i.e. no $\mathbb{Z}^{d}$-SFTs which factor onto every other
$\mathbb{Z}^{d}$-SFT. This follows easily from entropy considerations,
but the proof we give is of a recursive nature.

Fix $d$. It will be convenient to consider SFTs over alphabets which
are subsets of $\mathbb{N}$; this allows us to examine sets of SFTs
without restricting the alphabet size, and is no restriction since
any SFT is isomorphic to one over the alphabet N.

For a language $L$ over $\Sigma$, we say that a pattern $a\in\Sigma^{E}$
is $L$-admissible if, whenever $b\in\Sigma^{F}$ is a pattern in
$L$ and $F+u\subseteq E$, the pattern $b$ does not appear at $u$
in $a$; i.e. $a(u+v)\neq b(v)$ for some $v\in F$. 

We return to the question of $(d,d)$-universal SFTs. It is well-known
that the set \[
\mathcal{L}=\{L\,:\, L\mbox{ is a finite set of finite patterns over }\mathbb{N}\mbox{ and }X_{L}=\emptyset\}\]
is RE. To see this consider the algorithm that is given as input a
finite set $L$ over a finite $\Sigma\subseteq\mathbb{N}$, and iterates
over $n\in\mathbb{N}$; for each $n$ it checks if there exist $L$-admissible
$a\in\Sigma^{[-n;n]}$. If none exist it announces that $X_{L}=\emptyset$
and halts. Otherwise, is goes on to the next $n$. Clearly if $X_{L}\neq\emptyset$
then the algorithm will not halt, and a compactness argument shows
that if $X_{L}=\emptyset$ it will.

Thus, in order to prove the claim about non-existence of $(d,d)$-universal
SFTs, we show that, if there is some finite $L^{*}$ so that $X=X_{L^{*}}\subseteq\Sigma^{\mathbb{Z}^{d}}$
factors onto every $\mathbb{Z}^{d}$ SFT, then the set\[
\mathcal{M}=\{L\,:\, L\mbox{ is a finite set of finite patterns over }\mathbb{N}\mbox{ and }X_{L}\neq\emptyset\}\]
is RE. Since $\mathcal{L},\mathcal{M}$ are complementary in the space
of finite sets of patterns over $\mathbb{N}$, and both are RE, it
follows that they are recursive (lemma \ref{lem:RE-and-coRE-isR});
this would contradict Berger's theorem.

The following algorithm establishes that $\mathcal{M}$ is RE. As
input it accepts a finite set $L$ of patterns over $\mathbb{N}$
and decides if $X_{L}$ is empty. Recall that $L^{*}$ is assumed
to be a finite set of patterns so that $X_{L^{*}}$ factors onto every
$0$-entropy SFT; we shall use $L^{*}$ in constructing our algorithm.
Let $R\in\mathbb{N}$ be such that each pattern in $L^{*}$ and $L$
is supported in $[-R,R]^{d}$. 
\begin{algorithm*}
For each triple $r,k,\varphi$ with $r,k\in\mathbb{N}$, $r>R+k+1$,
and $\varphi:\Sigma^{[-k;k]^{d}}\rightarrow\Delta$, do
\begin{enumerate}
\item Enumerate all $L^{*}$-admissible patterns in $\Sigma^{[-r;r]^{d}}$.
Call them $a_{1},\ldots,a_{N}$
\item If $\varphi(a_{i})|_{[-R;R]^{d}}$ is $L$-admissible for every $a_{i}$,
$1\leq i\leq N$, announce that $X_{L}\neq\emptyset$ and halt.
\end{enumerate}
\end{algorithm*}
To see that this works, note that if the algorithm halts in (2) for
some triple $(r,k,\varphi)$ then the image under $\varphi$ of any
point in $X_{L^{*}}$ gives a point in $X_{L}$, implying that $X_{L}\neq\emptyset$.
Conversely, if $X_{L}\neq\emptyset$ then by universality of $X_{L^{*}}$
there is some $k$ and factor map $X_{L^{*}}\rightarrow X_{L}$ which
is defined locally by some $\varphi:\Sigma^{[-k;k]^{d}}\rightarrow\Delta$.
A compactness argument shows that for these $k,\varphi$, if the condition
in (2) fails for every $r$ then there is a point $x\in X_{L^{*}}$
with $\varphi(x)\notin X_{L}$, a contradiction.

We now deal with the general case.
\begin{proof}
(of theorem \ref{thm:main-negative}) The proof that there is no universal
EDS follows the same argument as above, the only difference being
that factor maps are no longer sliding block codes, which makes the
notation more cumbersome. If $(X,\mathbb{Z}^{d})$ is a totally disconnected
system and $\pi:X\rightarrow Y\subseteq\Sigma^{\mathbb{Z}^{d}}$ is
a factor map to a subshift $Y$, then the factor map is determined
by the partition \[
U_{\sigma}=\{x\in X_{L}\,:\, f(x)_{0}=\sigma\}\qquad\sigma\in\Sigma\]
of $X$ into closed and open sets, and conversely if $\{U_{\sigma}\}_{\sigma\in\Sigma}$
is such a partition then it defines a factor map $\pi$ into $\Sigma^{\mathbb{Z}^{d}}$,
where $(\pi x)(u)=\sigma$ if and only if $T^{u}x\in U_{\sigma}$.
Thus in order to adapt the proof above to general EDS we iterate over
partitions rather than sliding block codes; note that the finite partitions
of $\Omega$ into closed and open sets can be effectively enumerated.

Fix $d$ and suppose $L\subseteq\mathcal{P}^{*d}$ is an RE set and
that $\Omega_{L}$ factors onto every $\mathbb{\mathbb{Z}}^{d}$-EDS.
As before, we obtain a contradiction by showing that the set $\mathcal{L}'$
above is RE. 

Let $L_{n}$ be a recursive increasing sequence of sets with $\cup L_{n}=L$,
which exists since $L$ is RE. Consider the following algorithm, which
as input accepts a finite set $M$ of finite patterns over $\mathbb{N}$:
\begin{enumerate}
\item Let $R$ be an upper bound on the diameter of the patterns in $M$.
\item For each triple $r,n$ and $(U_{\sigma})_{\sigma\in\Sigma}$, with
$r,n\in\mathbb{N}$, $r>R$ and $(U_{\sigma})$ a closed and open
partition of $\Omega$, do:

\begin{itemize}
\item If for every $x\in\Omega$ the condition \[
\forall\left\Vert u\right\Vert \leq n\quad T^{u}x\notin\Omega\setminus\bigcup_{\overline{a}\in L_{n}}\bigcup_{\left\Vert u\right\Vert \leq n}T^{u}[\overline{a}]\]
implies that $f(x)|_{[R;R]^{d}}$ is $M$-admissible, announce that
$Y\neq\emptyset$ and halt.
\end{itemize}
\end{enumerate}
Although the condition in (a) is not phrased as a finite condition
and may appear hard to check, it is actually a question about the
non-emptiness of the intersection of finitely many cylinder sets which
are given in the data, and can therefore be effectively checked. The
proof that this algorithm produces the correct output is the same
as the one given above for SFTs, and we omit it.
\end{proof}

\section{\label{sec:Realization-of-effective-attractors}Realization of effective
attractors}

Let $U\subseteq\mathbb{R}^{d}$ be open and let $f:U\rightarrow U$
be an effective map with attractor $X\subseteq U$ as in definition
\ref{def:effective attractor}. We wish to show that there is an EDS
which extends $(X,f)$; theorem \ref{thm:realization-of-attractors}
then follows.

Let us say that a dyadic rational is one of the form $\frac{k}{2^{n}}$;
a dyadic interval is a closed interval of the form $[\frac{k}{2^{n}},\frac{k+1}{2^{n}}]$;
and by a dyadic cell is a product $I_{1}\times\ldots\times I_{d}$
of dyadic intervals. An open dyadic cell we mean the interior of a
dyadic cell, i.e. the product of open dyadic intervals. 

Given a binary representation $\overline{x}$ of $x\in\mathbb{R}$
let $D_{N}(\overline{x})$ denote the set of all $y\in\mathbb{R}$
whose first $N$ binary digits after the {}``decimal'' point agree
with $\overline{x}$; this is a closed dyadic interval of length $2^{-N}$.
For $x=(x_{1},\ldots,x_{d})\in\mathbb{R}^{d}$ and binary representations
$\overline{x}_{i}$ of $x_{i}$ we define $D_{N}(\overline{x})=D_{N}(\overline{x}_{i})\times\ldots\times D_{N}(\overline{x}_{d})$,
which is a dyadic cell with $x\in D_{N}(\overline{x})$. There are
at most $2^{d}$ binary representations of $(x_{1},\ldots,x_{d})$,
each giving rise to a dyadic cell containing $x$.
\begin{lem}
\label{lem:approximating-image-of-dyadic-cell}Let $f:U\rightarrow U\subseteq\mathbb{R}^{d}$
be an effective map. Then there is an algorithm which, given a closed
dyadic cell $D\subseteq U$ and an integer $n$, outputs a finite
set of rational points in $\mathbb{R}^{d}$ which are $\frac{1}{n}$-dense
in $f(D)$.\end{lem}
\begin{proof}
Let $A$ be the algorithm given by the definition of an effective
function; for a point $x\in\mathbb{R}^{d}$ with binary representation
$\overline{x}$ and for given $n$ let $A(x,n)$ denote the approximation
of $f(x)$ produced by $A$ on inputs $\overline{x},n$, so $\left\Vert A(\overline{x},n)-f(x)\right\Vert <\frac{1}{n}$.
Let $N=N(\overline{x},n)$ denote the number of digits from $\overline{x}$
that is used by $A$ in computing $A(\overline{x},n)$. If $x'\in D_{N}(\overline{x})$
and $\overline{x}'$ is a binary representation of $x'\in\mathbb{R}^{d}$
agreeing with $\overline{x}$ for the first $N$ digits, then $A(\overline{x}',n)=A(\overline{x},n)$,
since the algorithm halts before having a chance to detect that $\overline{x}\neq\overline{x'}$.
Hence, for $x'\in D_{N}(\overline{x})$ we have $\left\Vert f(x')-A(\overline{x},n)\right\Vert =\left\Vert f(x')-A(\overline{x}',n)\right\Vert <\frac{1}{n}$,
and so $\left\Vert f(x')-f(x)\right\Vert <\frac{2}{n}$.

There are at most $2^{d}$ binary representations $\overline{x}$
of $x$, and for each we get an $N$ as above. Let $N^{*}=N^{*}(\overline{x},n)$
be the maximum of these numbers and let $U_{n}(x)$ be the interior
of the union of dyadic cells of side $2^{-N^{*}}$ containing $x$.
This is an open set containing $x$, and from the discussion above
we see that the diameter of $f(U_{n}(x))$ is $<\frac{5}{n}$, since
the image of each cell has diameter $<\frac{2}{n}$ and is within
distance $\frac{1}{n}$ of $f(x)$. Note that for a dyadic point $r\in\mathbb{Q}^{d}$,
both $N^{*}(r,n)$ and $U_{n}(r)$ are computable.

If $D\subseteq U$ is a closed dyadic cell and $n\in\mathbb{N}$,
then there is a finite cover of $D$ by sets of the form $U_{n}(r)$,
and such a cover can be computed by iterating over all finite collections
of the sets $U_{n}(r)$ for $r$ a dyadic rational, until such a collection
is found that covers $D$. If $\{U_{n}(r_{i})\}_{i=1}^{M}$ is such
a collection, then as we have seen, the set $\{A(\overline{r_{i}},n)\}_{i=1}^{M}$
is $\frac{5}{n}$-dense in $f(D)$, where $A(\overline{r},n)$ is
the output of the algorithm on input $n$ and the binary representation
$\overline{r}$ of $r$.\end{proof}
\begin{lem}
\label{lem:coplement-of-attractor-is-RE}Let $X$ be the attractor
of an effective map $f:U\rightarrow U\subseteq\mathbb{R}$. Then there
is an algorithm which, when given a dyadic cell $D=I_{1}\times\ldots\times I_{d}\subseteq U$
as input, halts if $X\cap D=\emptyset$, and otherwise runs forever.\end{lem}
\begin{proof}
By assumption there is an open set $V\subseteq U$ with $\overline{fV}\subseteq V$
and $X=\cap f^{n}V$. By covering $X$ by small open dyadic cells
and then taking their closure, we see that there is a finite set of
closed dyadic cells $C_{1},\ldots,C_{N}$ which cover $X$ and are
contained in $V$. Denoting by $C$ their union, we have $X=\cap f^{n}C$.
Thus for any dyadic cell $D\subseteq U$ we have $D\cap X=\emptyset$
if and only if, for some $n$, $D\cap f^{n}C_{i}=\emptyset$ for $i=1,\ldots,N$.

We do not claim that $C$ can be found effectively but it exists and
can be described by finite data, and we may use it in the algorithm
that we now present. As input the algorithm takes a dyadic cell $D$.
It then iterates over $n$ and for each $n$ it computes a finite
set of points $F$ which are $\frac{1}{n}$-dense in $f(C)$ (this
can be done by the previous lemma). If every point in $f$ has distance
$>\frac{1}{n}$ from $D$ the algorithm halts; otherwise it proceeds
to the next $n$.
\end{proof}
In the same way we have
\begin{lem}
\label{lem:disjoint-images-is-RE}Let $X$ be the attractor of an
effective map $f:U\rightarrow U\subseteq\mathbb{R}$. Then there is
an algorithm which, given as input two (closed) dyadic cells $D',D''$,
halts $f(D')\cap D''=\emptyset$ and otherwise runs forever.
\end{lem}
We can now prove theorem \ref{thm:realization-of-attractors}:
\begin{proof}
Let $f:U\rightarrow U$ be an effective map with attractor $X\subseteq U$.
Since $X$ is bounded and we can assume that $U$ is bounded, and
without loss of generality $U\subseteq[0,1]^{d}$. Let $K=\{0,1\}^{\mathbb{N}}$
be the Cantor set and $\pi:K\rightarrow[0,1]^{d}$ be given by $x\mapsto(x_{1},\ldots,x_{d})$
where\[
x_{i}=\sum_{n=0}^{\infty}2^{-n+1}x(dn+i)\]
Let $Y\subseteq\Omega=K^{\mathbb{Z}}$ be those points $\omega$ so
that \[
\pi(\omega(n))\in X\quad,\quad n\in\mathbb{Z}\]
and\[
\pi(\omega(n+1))=f(\pi(\omega(n)))\]
Clearly $Y$ is a subshift and $\pi:Y\rightarrow X$ is a factor map
from $Y$ to $X$. It remains to show that $Y$ is an effective subshift. 

Recall the definition of $\mathcal{P}^{*1}$ from section \ref{sec:Recursive-set-and-effective-dynamics}.
Note that if $a\in\{0,1\}^{\{1,2\ldots dk\}}$ then $\pi([a])$ is
a dyadic cell in $[0,1]^{d}$. Consider the set of generalized cylinder
sets \[
L=\left\{ \overline{a}\in\mathcal{P}^{*1}\,\left|\:\begin{array}{c}
\overline{a}:\{n,n+1\}\rightarrow\{0,1\}^{\{1,2\ldots dk\}}\mbox{ for some }n,k\in\mathbb{N}\\
\mbox{and either }f(\pi([\overline{a}(n)]))\cap\pi([\overline{a}(n+1)])=\emptyset\\
\mbox{or }\pi([\overline{a}(n)])\cap X=\emptyset\mbox{ or }\pi([\overline{a}(n+1)])\cap X=\emptyset\end{array}\right.\right\} \]
One may verify that $Y=\Omega_{L}$ ; the proof will be complete if
we show that $L$ is RE. In order to so this it suffices to show that
there is an algorithm which, given $\overline{a}:\{n,n+1\}\rightarrow\{0,1\}^{\{1,2\ldots dk\}}$,
halts if $\pi([\overline{a}(n)])\cap X=\emptyset$ or $\pi([\overline{a}(n+1)])\cap X=\emptyset$
or $f(\pi([\overline{a}(n)]))\cap\pi([\overline{a}(n+1)])=\emptyset$,
and otherwise runs forever. But this is a direct consequence of the
lemmas preceding this proof.
\end{proof}

\section{\label{sec:Open-Problems}Open Problems}

We conclude with a few problems which arise in connection with this
work. 

We have shown that universality cannot occur for effective $\mathbb{Z}^{d}$
systems. This can be rephrased as follows: If $d>2$ then for every
effective $\mathbb{Z}^{d}$-system $X$, there is a $\mathbb{Z}^{d}$-SFT
$Y$ so that $X$ does not factor onto $Y$. If we relax the effectiveness
requirement, a stronger result is true: under mild assumptions on
the original system $X$, there is a system disjoint from $X$ in
the sense of Furstenberg. 
\begin{problem*}
Given a minimal $\mathbb{Z}^{d}$-EDS, does there exists a $\mathbb{Z}^{d}$-EDS
disjoint from it? (one may of course ask this about other classes).
\end{problem*}
(minimality is required to avoid trivial counterexamples). The same
question may be asked for minimal SFTs.

The non-existence of universal EDS was demonstrated using purely recursion-theoretic
considerations. It is not clear how to handle some restricted interesting
classes of EDS. Of particular interest are the minimal SFTs, nontrivial
examples of which were constructed in \cite{M89} and more examples
follow from \cite{HM07}. We note that for SFTs, minimality implies
zero entropy, so entropy cannot rule out a universal minimal SFT.
We also note that, besides being an interesting class dynamically,
minimal SFTs have the additional feature that the set of patterns
appearing in a minimal SFT is recursive, and the extension problem
is decidable for them, i.e. given a locally admissible pattern one
can decide if it can be extended to an infinite configuration \cite{H07}.
Thus this class does not exhibit the recursive complexity of SFTs
in general. 
\begin{problem*}
Are there universal systems in the class of minimal SFTs? (and if
so, in what dimensions?)
\end{problem*}
Finally, a recursive product of effective systems is effective. One
can ask a related question about SFTs: 
\begin{problem*}
If a recursive product SFTs has finite entropy, can it be extended
to an SFT?
\end{problem*}
\bibliographystyle{plain}
\bibliography{bib}

\end{document}